\documentclass[leqno,10pt]{amsart}
\usepackage{latexsym,amssymb,amsmath,graphics,color}
\usepackage[dvips]{graphicx}

\def\E{{\mathbb{E}}}

\def\N{{\mathbb{N}}}
\def\P{{\mathbb{P}}}

\def\R{{\mathbb{R}}}

\def\T{{\mathcal{T}}}

\def\Z{{\mathbb{Z}}}
\def\PP{{\mathbb{P}}}

\newcommand{\wt}{\widetilde}

\def\8{\infty}
\def\N{\mathbb{N}}
\def\E{\mathbb{E}}
\def\P{\mathbb{P}}

\def\<{\langle}
\def\>{\rangle}

\renewcommand{\d}{\delta}
\renewcommand{\a}{\alpha}
\renewcommand{\b}{\beta}

\newcommand{\e}[1]{\E\left[#1\right]}
\newcommand{\eps}{\varepsilon}

\newcommand{\g}{\gamma}
\newcommand{\1}{\bold 1}

\newcommand{\pr}[1]{\PP\left[#1\right]}

\renewcommand{\1}[1]{{\bf 1}_{\left[#1\right]}}

\newtheorem{cor}[equation]{Corollary}
\newtheorem{corollary}[equation]{Corollary}
\newtheorem{lem}[equation]{Lemma}
\newtheorem{lemma}[equation]{Lemma}

\newtheorem{theorem}[equation]{Theorem}
\newtheorem{remark}[equation]{Remark}

\newtheorem{prop}[equation]{Proposition}
\theoremstyle{definition}

\numberwithin{equation}{section}

\begin{document}

\title[Linear stochastic equations]{Linear stochastic equations in the critical case}
\author[D. Buraczewski, K. Kolesko]{Dariusz Buraczewski and Konrad Kolesko}
\address{D. Buraczewski, K. Kolesko\\ Instytut Matematyczny\\ Uniwersytet Wroclawski\\ 50-384 Wroclaw\\
pl. Grunwaldzki 2/4\\ Poland}
\email{dbura@math.uni.wroc.pl\\ Konrad.Kolesko@math.uni.wroc.pl}

\begin{abstract}
We consider solutions of the stochastic equation
$X \stackrel{d}= \sum_{i=1}^N A_iX_i + B$, where $N$ is a random natural number, $B$ and $A_i$ are random positive numbers and $X_i$ are independent copies of $X$, which are independent also of $N,B,A_i$. Properties of solutions of this equation are mainly coded in the function $m(s)=\mathbb{E}\big[ \sum_{i=1}^N A_i^s \big]$. In this paper we study 
the critical case when the function $m$
 is tangent to the line $y=1$. Then, under a number of further assumptions, we prove existence of solutions and  describe their asymptotic behavior.
\end{abstract}

\thanks{ D. Buraczewski was partially supported by MNiSW grant N N201 393937.  K. Kolesko was partially supported by MNiSW grant N N201 610740}


 \maketitle

  \section{Introduction}
The main purpose of the present paper is to study a class of  linear  stochastic equations, existence of their solutions and to describe properties of those solutions. The simplest example we have in mind is the random difference equation, called often also a first order random coefficients
autoregressive model,
\begin{equation}
\label{random difference}
X \stackrel{d}= AX+B,
\end{equation}
where all the random variables are real valued, $X$ is independent of the pair $(A,B)$, and the sign '$\stackrel{d}=$' denotes equality in distribution. It is well-known that the equation above has a unique solution if $\E\big[ \log A\big]<0$ and $\E\big[\log^+|B|\big]<\8$. The celebrated Kesten theorem \cite{Kesten} says that if a random variable $X$ is a solution of the equation \eqref{random difference}, then under a number of assumptions, the main being existence of a positive $\a$ such that $\E [A^\a] = 1$, the random variable $X$ is $\a$-regularly varying, i.e.
$$
\lim_{x\to \8} x^\a \P[X>x] = C_+,
$$
for some positive constant $C_+$ (see also the paper of Goldie \cite{Goldie} for a transparent and elegant proof). Since the random
difference equation appears both in many applied models e.g. in financial mathematics and in purely mathematical problems, the last result found enormous number of applications in the literature.

\medskip

In this paper we consider  general linear  stochastic equations, i.e. equations of the form
 \begin{align}
  \label{rownanie:galazkowe}
  X\stackrel d=\sum_{i=1}^NA_iX_i+B,
 \end{align}
 where  $X$, $X_i$ are i.i.d. and  independent of $(N,B,A_1,A_2\cdots)$. Notice that the  last  formula depends only on $N$ first values of $A_i$'s, therefore without any loss of generality, we assume that $A_i=0$ for $i>N$. Moreover in this paper we restrict our attention to positive random variables, i.e. we assume that $X_i$, $A_i$ and $B$ are positive.

 Equation  \eqref{rownanie:galazkowe} is also called the inhomogeneous smoothing transform and the explanation of this name is the following. Given $\mu\in\mathcal{P}(\R^+)$ we define $\wt T\mu$ as the law of $\sum_{i=1}^N A_iX_i+B$, where $X_1,X_2,\dots$ is i.i.d. sequence with distribution $\mu$, independent of the vector $(N, B,A_1,A_2,\dots)$.
 Then any fixed point of $\wt T$ is characterized as a distribution of a random variable $X$ that satisfies \eqref{rownanie:galazkowe}.

This equation has gained importance in the last few years, since it turns out to be closely related to  important objects in the computer science: the Quicksort algorithm \cite{NR,RR} (and other divide and conquer algorithms), the Pagerank algorithm \cite{VL,JO1,JO2} (being in the heart of the Google engine) and in stochastic geometry \cite{PW}. The inhomogeneous equation was recently used to describe equilibrium
distribution of a class of kinetic models see e.g. \cite{BLT}

\medskip

Equation \eqref{rownanie:galazkowe} is also a generalization of the homogeneous smoothing transform, which is defined exactly as above but without the inhomogeneous term $B$, i.e. with $B=0$ a.s. Thus, we say that $X$ is a solution (or a fixed point) of a homogeneous smoothing transform if
 \begin{align}
  \label{rownanie:galazkowe:jednorodne}
  X\stackrel{d}=\sum_{i=1}^NA_iX_i,
 \end{align}
 where $X_1,X_2,\dots$ are independent copies of $X$ and the vector $(N, A_1,A_2,..)$ is independent of the sequence $\{X_i\}$. The last equation appeared in the literature already in the eighties in connection with studying interacting particle system \cite{DL}. It turned out also to have a number of applications e.g. in branching random walks  \cite{HS}.

Existence of solutions of \eqref{rownanie:galazkowe:jednorodne} and their properties were deeply studied in \cite{DL,Liu} (see also the recent paper \cite{ABM}) and it turns out
 that their properties are encoded in the function
 \begin{align}
  \label{funkcja:teta}
   m (t)=\e{\sum_{i=1}^NA_i^t}.
 \end{align}
Notice, equation \eqref{rownanie:galazkowe:jednorodne} does not  have a  unique solution since $tX$ for $t\in \R$ also solves it as long as $X$ does.
  We summarize known results in the following Lemma
 \begin{lemma}[\cite{DL,Liu}]
 \label{dl}
 If  $\E N >1$ and $\inf_{s\in [0,1]} m(s)\le 1$, then the set of solutions of \eqref{rownanie:galazkowe:jednorodne} is nonempty. Moreover if  $\e{\left(\sum_{i=1}^NA_i\right)^{1+\d}}<\8$, $\e{N^{1+\d}}<\8$ and for some $\a\in(0,1)$ we have $m(\a)=1$, $m'(\a)\le0$ then
  \begin{align*}
   &\lim_{x\to\8}\pr{X>x}x^{\a}=c\qquad\text{ if }m'(\a)<0,\\
   &\lim_{x\to\8}\pr{X>x}x^{\a}/\log x=c\qquad\text{ if }m'(\a)=0,
  \end{align*}
for some positive constant $c$.
 \end{lemma}

We begin the study of the nonhomogeneous smoothing transform explaining how to construct a solution to equation \eqref{rownanie:galazkowe} (see \cite{AM, AM2} for more details).
 Let $\T=\bigcup_{k\ge0}\N^{k}$ be an infinite Ulam-Harris tree, where $\N^{0}=\{\emptyset\}$.
 For $v=(i_1,\dots,i_n)$ we define the length $|v|=n$ and by $vi$ we denote the vertex $(i_1,i_2,\dots, i_n,i)$. We write $u<v$ if $u$ is a proper prefix of $v$, i.e. $u=(i_1,..,i_k)$ for some $k<n$. Moreover we write $u\le v$ if $u<v$ or $u=v$.
 Now we take $\{(B(v),A_1(v),A_2(v),\dots)\}_{v\in\T}$ a family of i.i.d. copies of $(B,A_1,A_2,\dots)$ indexed by the vertices of $\T$. For $v\in \T$ we also define a random variable $L(\emptyset)=1$ and $L(vi)=L(v)A_i(v)$. We can define now
 \begin{align}
  \label{rozwiazanie:R}
  R=\sum_{v\in\T}L(v)B(v).
 \end{align}
 One can easily   check that if the  series above is  finite almost surely then the random variable $R$ satisfies \eqref{rownanie:galazkowe}.
 However also the converse is true. Alsmeyer and Meiners \cite{AM} proved that existence of solutions of \eqref{rownanie:galazkowe} is equivalent to finiteness of the series \eqref{rozwiazanie:R}. Knowing that there exists one solution, one can construct a whole family of solutions just adding to $R$ any $Y$ being a solution of \eqref{rownanie:galazkowe:jednorodne}. However $R$ is distinguished by the property that it is the minimal solution (in the sense of stochastic domination i.e. $\pr{R>t}\le\pr{X>t}$ for any other solution $X$), see \cite{AM} for more details. Another useful property is that $R$ is the unique solution that is measurable with respect to the input data ${(B(v),A_1(v),A_2(v),\dots)}_{v\in\T}$ (compare with \cite{Aldous_Ban} where it is called an endogeneous solution). From now we  call $R$ the minimal solution.
 Therefore, if we can describe the tail of $R$, in view Lemma \ref{dl}, we obtain a full description of tails of all solutions \eqref{rownanie:galazkowe}.

 Similarly like in the homogeneous case the fundamental role in description of solutions of \eqref{rownanie:galazkowe} plays the function $m$ defined in \eqref{funkcja:teta}.  The necessary condition ensuring finiteness of  \eqref{rozwiazanie:R} is that $ m (t_0)\le1$ for some $t_0\in[0,1]$. However sufficient conditions are  still not established. It is known \cite{AM, JO1} that if $m(s)<1$ for some $s\in (0,1)$ and $\E B^s<\8$ then $R$ is well defined.

 Jelenkovi{\'c} and Olvera-Cravioto \cite{JO1,JO2} proved that $R$  has a power law distribution:
 \begin{lemma}
  \label{thm:JO}
  Let $(B,A_1,A_2,\dots)$ be a nonnegative random vector, with $N\in\N\cup\{\8\}$, $\pr{B>0}>0$ and $R$ be the minimal solution to \eqref{rownanie:galazkowe} given by \eqref{rozwiazanie:R}. Suppose that
  \begin{itemize}
  \item
  the equation $m(s)=1$ has 2 solutions $\a<\b$;
   \item $\e{B^{\b}}<\8$, $0<m'(\b)=\e{\sum_{i=1}^NA_i^{\b}\log A_i}<\8$.
   \item there exists $j\ge1$ with $\pr{N\ge j, A_j>0}>0$ such that the measure $\pr{\log A_j\in du, A_j>0, N\ge j }$ is nonarithmetic;
   \end{itemize}
   In addition, assume that
  \begin{itemize}
   \item[a)] $m(1)=\e{\sum_{i=1}^NA_i}<1$ and $\e{\left(\sum_{i=1}^NA_i\right)^{\b}}<\8$, if $\b>1$;
   \end{itemize}
   or
   \begin{itemize}
   \item[b)] $\e{\left(\sum_{i=1}^NA_i^{\b/(1+\eps)}\right)^{1+\eps}}<0$ for some $0<\eps<1$, if $0<\b\le1$.
  \end{itemize}
  Then,
  $$\pr{R>t}\sim Ct^{-\b},\qquad t\to\8,$$
  for some $C>0$.

 \end{lemma}
 \begin{remark}
  Positivity of the constant $C$ was not discussed in \cite{JO1,JO2} and was proved recently in \cite{BDZ}.
 \end{remark}

 Summarizing  if  $\a<\b$ are two solutions of the equation $m(s)=1$ and $\a<1$, then  the minimal solution $R$ of \eqref{rownanie:galazkowe} has a power law of index $\b$. Any other solution $X$ of \eqref{rownanie:galazkowe} satisfies $\pr{X>t}\sim Ct^{-\a}$.

\medskip

The main purpose of the present paper is to  complete the picture and to  study the critical case, when the equation $m(s)=1$ has exactly one solution $\a < 1$ and then $m'(\a)=0$, i.e. when the graph of the function $m(s)$ is tangent to the line $y=1$. For the random difference equation \eqref{random difference} this corresponds to the situation when the graph of the Mellin transform $\E [A^s]$ is tangent to the line $y=1$ at 0, i.e. when $\E[\log A]=0$. Then it is known that equation \eqref{random difference} has no solutions, nevertheless when written in terms of measures has solutions in the class of Radon measures on $\R$. Existence and asymptotic properties of solutions were studied in \cite{BBE, BBD, Bura}. For the homogeneous smoothing transform the critical case was considered by Durrett, Liggett \cite{DL} and Liu \cite{Liu} and is a part of Lemma \ref{dl} (see also \cite{bur-spa} for the case $\a=1$).

The only result we know concerning the inhomogeneous smoothing transform in the critical case is due to Alsmeyer and Meiners \cite{AM}, who proved that for $\a<1/5$ (and under some further assumptions) the series \eqref{rozwiazanie:R} is finite, providing thus a solution to \eqref{rownanie:galazkowe}.

The main result of this paper is the following
 \begin{theorem}
 \label{thm:twierdzenie.asymptotyka}
  Suppose that
 \begin{align}
  \label{istnieje:alpha}
  &\mbox{There exists }0<\a<1\mbox{ such that }m(\a)=\e{\sum_{i=1}^NA_i^{\a}}=1,\\
  \label{stycznosc:m}
  &m'(\a)=\e{\sum_{i=1}^NA_i^{\a}\log A_i}=0,\\
  & \E [N] >1,\\
  \label{zalozenie:nieartmetycznosc}
  &\mbox{For some }j\mbox{ the measure }\pr{\log A_j\in du, A_j>0, N\ge j }\mbox{ is nonarithmetic},\\
  \label{zalozenia.momenty}
  &\e{N^{1+\d}+B^{\a+\d}+\sum_{i=1}^N\left(A_i^{-\d}+A_i^{\a+\d}\right)}<\8,\mbox{ for some }0<\d<1-\a.
 \end{align}
Then the minimal solution $R$ of \eqref{rownanie:galazkowe} is well defined and moreover
  \begin{equation}
   \lim_{t\to \8} t^{\a}\pr{R>t}=C_+
  \end{equation}
and the constant $C_+$ is strictly positive.
 \end{theorem}

Thus in the critical case the tail of the minimal solution of \eqref{rownanie:galazkowe} is of the order $t^{-\a}$, whereas the tails of all the other solutions behaves at infinity like $\log t\ t^{-\a}$.

\medskip

We finish the introduction with an overview over the organization of the paper. In Section \ref{istnienie:galazek} we prove that $\P[R>t]\le C t^{-\a}$, that in particular implies that $R$ is finite a.s. The most essential part of the proof is contained in Section \ref{section3}. We reduce the problem to study behavior at 0 of the Laplace transform $\phi$ of $R$. Considering $\phi$ as a solution of the Poisson equation we first prove that it behaves regularly at 0 (Section \ref{subsection3.1}) and the deduce the correct asymptotic (Section \ref{subsection3.3}). Finally, applying some arguments based on the Landau theorem and holomorphic functions, we prove positivity of the limiting constant (Section \ref{subsection3.4}).

\medskip

The authors are grateful to Jacek Zienkiewicz for stimulating discussions on the subject of the paper.

 \section{Existence of a solution and its first estimates}
 \label{istnienie:galazek}
 In this section we prove 

 \begin{prop}
 \label{prop:existence}
 Assume hypotheses of Theorem \ref{thm:twierdzenie.asymptotyka} are satisfied,
  then
 $$ \P[R>t] \le C t^{-\a}, $$ where $R$ is the random variable defined in \eqref{rozwiazanie:R}.
 In particular $R$ is finite a.s.
 \end{prop}
 \begin{cor}
 \label{cor:moments}
 $\E[ R^\b]$ is finite  for every $\b<\a$.
 \end{cor}
We start with the following lemma

\begin{lemma}
 \label{lem:1} Let $\{Y_n\}_{n\in \N}$ be a sequence of i.i.d. random variables such that $\E Y_1 =0$. 
 Let $S_n = \sum_{i=1}^n Y_i$ be the sequence of the partial sums. Then, for any strictly positive constant $\d$, the function $$W(x)=\e{\sum_{i=0}^{\8}e^{-\d (x+S_i)}\1{S_j+x\ge 0\mbox{ for }j\le i}},$$
 is bounded.
\end{lemma}
\begin{proof}
By definition the function $W$ can be nonzero only for positive $x$. Let $L=\inf\{i:S_i<0\}$, then
 \begin{align*}
  W(x)&=\e{\sum_{i=0}^{\8}e^{-\d(x+S_i)}\1{S_j+x\ge 0\text{ for }j\le i}}\\
  &=\e{\sum_{i=0}^{L-1}e^{-\d(x+S_i)}\1{S_j+x\ge 0\text{ for }j\le i}} +\e{\sum_{i=L}^{\8}e^{-\d(x+S_i)}\1{S_j+x\ge0\text{ for }j\le i}}\\
  &=\e{\sum_{i=0}^{L-1}e^{-\d(x+S_i)}}\1{x\ge 0} +\e{\sum_{i=L}^{\8}e^{-\d(x+S_i)}\1{S_j+x\ge0\text{ for }L\le j\le i}}\\
  &=\e{\sum_{i=0}^{L-1}e^{-\d S_i}}e^{- \d x}\1{x\ge 0} +\e{W(x+S_L)}.
 \end{align*} Notice that the first expression above is just a finite constant, since by the reflection principle \cite{Feller}
$$  \e{\sum_{i=0}^{L-1}e^{-\d S_i}}=\e{\sum_{n=0}^{\8}e^{-\d S_{T_n}}}=:C<\8,$$
where $T_n$ is the sequence of upward ladder times:
  $T_0=0$, $T_n=\inf\{i>T_{n-1}:S_i\ge S_{T_{n-1}} \}$.
Moreover the function $f(x)=e^{-\d x}\1{x\ge 0}$ is directly Riemann integrable (dRi), i.e. it is integrable and satisfies
\begin{equation}
\label{dri}
  \lim_{h\to0}h\sum_{n\in \Z}\sup_{ x,y\in I_n(h)}|f(x)-f(y)|=0,
 \end{equation}
 where $I_n(h)=(nh,(n+1)h]$.

 Finally, we can express $W$ as the convolution  of the function $f$ with the  potential of the  transient random walk $V_n$, where $V_n$ is the sum of $n$ independent copies of $S_L$. Therefore, independently of $x$,  we have
 \begin{align*}
  W(x)&=\sum_{n=0}^{\8}f(x+V_n)\le C\sum_{n=0}^{\8}\sum_{k=0}^{\8}e^{-\d k}\1{x+V_n\in[k,k+1)}\\
  &\le C\sum_{k=0}^{\8}e^{-\d k}\sum_{n=0}^{\8}\pr{x-k+V_n\in[0,1)}\le C'\sum_{k=0}^{\8}e^{-\d k}<\8,
 \end{align*}
 where the uniform bound in the last line follows from Proposition 2.1 in \cite[CH. 5]{Revuz}.

\end{proof}

  Let us introduce a random variable $Y$ with distribution  given by
 \begin{align}
  \label{rozklad:Y}
  \e{f(Y)}=\e{\sum_{i=1}^N f(-\log A_i)A_i^{\a}},
 \end{align}
 for any positive Borel function $f$. By \eqref{istnieje:alpha}, the right hand side indeed defines a probability measure.
 The main properties of $Y$, we are going to use, are summarized in the following lemma
 \begin{lem}
 \label{lemma Y}
 The random variable $Y$ is centered  ($\E Y = 0$), nonarithmetic (the closed subgroup generated by the support of the measure $\pr{Y\in dx}$ is $\R$) and has finite exponential moments
  $$   \e{e^{\pm \d Y}} <\8
  $$ for some $\d>0$.
 \end{lem}
 \begin{proof}
 We have
 \begin{align*}
  &\e{e^{\pm \d Y}}=\e{\sum_{i=1}^NA_i^{\a\mp\d}}=m(\a\mp\d)<\8,\\
  &\e{Y}=\e{\sum_{i=1}^N-A_i^{\a}\log A_i}=0.
 \end{align*}
 Nonarithmecity follows from  assumption \eqref{zalozenie:nieartmetycznosc}.
\end{proof}

\begin{proof}[Proof of Proposition \eqref{prop:existence}]
We compare behavior of the sum $R = \sum_{v\in\T}L(v)B(v)$ with behavior of the maximum $\wt R = \max_{v\in \T} L(v)$.

We first  prove that
\begin{equation}
\label{estimates of max}
\P[\wt R > t] \le C t^{-\a},
\end{equation}  for some positive constant $C$.

\medskip

Let $\{Y_i\}$ be a sequence of independent copies of $Y$ defined in \eqref{rozklad:Y} and let $S_n$ be the sequence of their partial sums.
Applying the definition of $Y$ and reasoning by the induction (see e.g.  \cite{AS}) one can easily prove the following well-known formula:
  \begin{equation}
  \label{eq: zamiana}
   \e{e^{\a S_n}f(S_1,...,S_n)}=\e{\sum_{|v|=n}f(-\log L(v_1),...,-\log L(v_n))},
  \end{equation}
  valid for a fixed $n$ and any test function $f:\R^n\to \R$.

Putting $f(x_1,..,x_n)=\1{x_1\ge -\log t,...,x_{n-1} \ge -\log t,x_n<-\log t}$ we obtain
  \begin{eqnarray*}
   \pr{\wt R>t}&=&\pr{L(v)>t, \mbox{ for some }v\in\T}\\
   &=&\sum_n\pr{L(v)>t \mbox{ for some } |v|=n \mbox{ and }L(u)\le t \mbox{ for } u\le v}\\
   &\le&\sum_n \e{\sum_{|v|=n}\1{-\log L(v_1)\ge -\log t,...,-\log L(v_{n-1}) \ge -\log t,-\log L(v)<-\log t}}\\
   &=&\sum_n\e{e^{\a S_n}\1{S_1 \ge -\log t,...,S_{n-1}\ge -\log t,S_n<-\log t}}\\ &\le& t^{-\a},
  \end{eqnarray*}
hence we obtain \eqref{estimates of max}.

\medskip

Next  we write

$$
 \pr{R>t}\le \P\big[\wt R>t\big]+\P\Big[\big\{ R>t \big\} \cap \big\{\wt R\le t\big\}\Big].
$$
In view of \eqref{estimates of max} it is sufficient to estimate only the second term. Taking $\g = \a+\d<1$, we have
\begin{eqnarray*}
 \P\Big[\big\{ R>t \big\} \cap \big\{\wt R\le t\big\}\Big]
 &\le& \P \bigg[ \sum_{v\in\T}L(v)B(v)\1{L(v')\le t\text{ for }v'\le v}>t\bigg] \\
 &\le& \P \bigg[ \sum_{v\in\T}L(v)^\g B(v)^\g \1{L(v')\le t\text{ for }v'\le v}>t^\g\bigg] \\
 &\le& \frac 1{t^\g} \E\big[ B^\g \big]  \E\bigg[\sum_{v\in\T}L(v)^\g \1{L(v')\le t\text{ for }v'\le v}\bigg]
\end{eqnarray*}
Applying again \eqref{eq: zamiana} we obtain
\begin{eqnarray*}
 \E\bigg[\sum_{v\in\T}L(v)^\g\1{L(v')\le t\text{ for }v'\le v}\bigg]
 &=&\sum_{n}\E\bigg[\sum_{|v|=n}L(v)^\g\1{L(v')\le t\text{ for }v'\le v}\bigg]\\
 &=&\sum_{n}\e{e^{\a S_n}e^{-\g S_n}\1{S_k+ \log t\ge 0\text{ for }k\le n}}\\
 &=&\sum_{n}\e{e^{-(\d) (S_n+\log t)}t^{\d}\1{S_k+\log t\ge 0 \text{ for }k\le n}}\\
 &=& t^{\d}W(\log t),
 \end{eqnarray*}
where $W$ is the function defined in Lemma \ref{lem:1}, which as we already know is bounded.
Finally, since $\E[B^\g]<\8$, we obtain
$$
 \P\Big[\big\{ R>t \big\} \cap \big\{\wt R\le t\big\}\Big] \le C t^{-\a}.
$$
\end{proof}

 \section{Tail of the solution}
 \label{section3}
 \subsection{The Poisson equation}
 \label{subsection3.1}

 For a non-negative random variable $X$ by $\phi_{X}(t)=\e{e^{-tX}}$ we denote its Laplace transform. For simplicity we write  $\phi=\phi_{R}$ for the Laplace transform of $R$.
To prove our main result we  use the duality between the tail behaviour of $R$ and the behaviour of its Laplace transform $\phi$ near 0
 given by the following Tauberian theorem (its proof can be found e.g. in the book of Feller \cite{Feller}, Example c) after Theorem 4 in Chap. XIII).
  \begin{lemma}[Tauberian Theorem]
 \label{thm:tauberowskie}
   For $0<\a<1$ and  a slowly varying function $L$ the following are equivalent:
 \begin{align*}
  i)&\qquad\lim_{x\to\8}\frac{x^{\a}\pr{R>x}}{L(x)}=1\qquad\text{}\hspace{150pt}\\
  ii)&\qquad\lim_{t\to0}\frac{1-\phi(t)}{t^{\alpha}L(1/t)}=\Gamma(1-\a).
 \end{align*}
 \end{lemma}
Thus, in order to describe the tail of $R$, i.e. $\P[R>t]$, it is sufficient to study its Laplace transform $\phi$ and prove
$$
\lim_{t\to 0}\frac {1-\phi(t)}{t^\a} = C.$$ 
It is convenient for our purpose to change the coordinates and define
$$
D(x)=e^{\a x}(1-\phi(e^{-x})).
$$ Then our aim is to prove
\begin{equation}
\label{eq: granica D}
\lim_{x\to \8} D(x) = C.
\end{equation}
We will often use the following well-known lemma
 \begin{lemma}
 \label{malenie:transformaty} For any positive random variable $X$ and
 any $0<\g<1$ we have
  $$1-\phi_X(t)\le \Gamma(1-\g)\e{X^\g}t^{\g}.$$
 \end{lemma}
 \begin{proof}
 Notice that $\e{f(X)}=\int_0^{\8}f'(s)\pr{X>s}ds$ for nonnegative $X$ and any differentiable, monotone function $f$ such that $f(0)=0$. Then by Chebyshev inequality
  \begin{align*}
   1-\phi_X(t)&= \e{1-e^{-tX}}=\int_0^{\8}(1-e^{-s})'\pr{tX>s}ds\\
   &\le\int_0^{\8}e^{-s}\e{(tX)^{\g}}s^{-\g}ds= \Gamma(1-\g)\e{X^\g}t^{\g}.
  \end{align*}
 \end{proof}

To prove \eqref{eq: granica D} we apply some techniques described in the paper of Durrett and Liggett \cite{DL}, who considered solutions of the homogeneous equation \eqref{rownanie:galazkowe:jednorodne}. We adopt their ideas, however it turns out that adding the additional term $B$ causes serious problems, hence we will present here all the details of the proof.

\medskip

Let $Y$ be the random variable defined in \eqref{rozklad:Y}.
We  consider $D$ as a solution of the Poisson equation.
 \begin{lemma}
The function $D$ satisfies the following Poisson equation
    $$\e{D(x+Y)}=D(x)+G(x),$$  where
    $$G(x)=e^{\a x}\e{\sum_{i=1}^N (1-\phi(e^{-x} A_i))-\left(1-e^{-e^{-x}B}\prod_{i=1}^N \phi(e^{-x}A_i)\right)}.$$
 \end{lemma}

 \begin{proof}
Notice first that  rewriting equation \eqref{rownanie:galazkowe} in terms of Laplace transform $\phi$ we obtain
$$
  \phi(t)=\e{\prod_{i=1}^N\phi(tA_i)e^{-tB}}.
$$
 Hence by
  the definition of $D$ and the  equation above  we have
   \begin{multline*}
  \e{D(x+Y)}-D(x)
  =\e{e^{\a (x+Y)}\left(1-\phi\left(e^{-(x+Y)}\right)\right)}-e^{\a x}(1-\phi(e^{-x}))\\
  =e^{\a x}\E\Bigg[\sum_{i=1}^N e^{-\a \log A_i}\left(1-\phi\left(e^{-x+\log A_i}\right)\right)A_i^{\a}
  -\left(1-e^{-e^{-x}B}\prod_{i=1}^N \phi(e^{-x}A_i)\right)\Bigg]\\
  =e^{\a x}\e{\sum_{i=1}^N (1-\phi(e^{-x} A_i))-\left(1-e^{-e^{-x}B}\prod_{i=1}^N \phi(e^{-x}A_i)\right)} =G(x)
 \end{multline*}
 \end{proof}

 We need also the following technical lemma  saying that for any $t\in(0,\a+\d)$ the sum ${\sum_{i=1}^NA_i^t}$, that appears under the expected value in  the definition of $m$ \eqref{funkcja:teta}, has moment bigger than 1.
 \begin{lemma}
  \label{zabawa:z:momentami}
  Assume that  $\e{N^{1+\d}}<\8$ for some $\d>0$. Let $X_i$ be an arbitrary sequence of random variables. Then  for any $r>1$
  and any $p\in\big(1,\frac{r(1+\d)}{r+\d}\big)$ we have
  $$
  \e{\left(\sum_{i=1}^N X_i^{1/r}\right)^p}\le C_{r,p} \e{\sum_{i=1}^NX_i},
  $$
where $C_{r,p}$ is a constant depending on $r$ and $p$.
 \end{lemma}
 \begin{proof}
 We will use the following well-known inequality (being just a simple consequence of the H\"older inequality)
  $$\left(\sum_{i=1}^na_i\right)^r\le n^{r-1}\sum_{i=1}^na_i^r,$$
  where $a_i\ge0$.

   Applying first the  last inequality and then the H\"{o}lder inequality with parameters ${r}/{p}$ and $(r/p)'$ (given $q>1$, we denote by $q'$  the conjugate real number such that $1/q+1/q'=1$) we obtain 
 $$  \e{\left(\sum_{i=1}^NX_i^{1/r}\right)^p}
   \le\e{N^{(r-1){p}/{r}}\left(\sum_{i=1}^NX_i\right)^{{p}/{r}}}
   \le\e{N^{\frac{(r-1)p}{r}\left(\frac{r}{p}\right)'}}^{1/(r/p)'}\e{\sum_{i=1}^NX_i}^{p/r}.$$
  Notice that since
  \begin{align*}
   (r-1)\frac pr\left(\frac rp\right)'\le1+\d
  \end{align*}
  in view of our assumptions both expressions above are finite.
 \end{proof}

We prove now a weaker result than  \eqref{eq: granica D} saying that the function $D$ behaves regularly at infinity.
\begin{prop}
\label{prop: regularnosc D}
Assume that hypotheses of Theorem \ref{thm:twierdzenie.asymptotyka} are satisfied. Then for  any $y\in \R$ we have
$$
\lim_{x\to\8}\frac{D(x+y)}{D(x)}=1.
$$
\end{prop}
\begin{proof}
We divide the proof of the proposition into several steps. Assume first additionally that the following condition is satisfied
\begin{align}
  \label{dodatkowe.zalozenie}
  \lim_{t\to 0}\frac{1-\e{e^{-tB}}}{1-\phi(t)}=0
  \end{align}

\medskip

{\bf Step 1. }  First we will show that
\begin{equation}
\label{eq: step1}
\lim_{x\to\8}\frac{G(x)}{D(x)}=0.
\end{equation}

  We can write
  \begin{align}
  \nonumber
   \lim_{x\to\8}\frac{G(x)}{D(x)}&=\lim_{x\to\8}\e{\frac{\prod_{i=1}^N\phi(e^{-x}A_i)(e^{-e^{-x}B}-1)}{1-\phi(e^{-x})}}\\
   \label{drugi:skladnik}
   &+\lim_{x\to\8}\e{\frac{\prod_{i=1}^N\phi(e^{-x}A_i)-1+\sum_{i=1}^N(1-\phi(e^{-x}A_i))}{1-\phi(e^{-x})}}
  \end{align}

  By our assumptions the first term  is equal to 0 since the Laplace transform $\phi$ is bounded by 1.

  In order to show that the second limit \eqref{drugi:skladnik} is zero we will use the following inequality valid for $0\le u_i\le v_i\le1$
  \begin{align}
   \label{DL:nierownosc1}
   \prod_{i=1}^Nu_i-1+\sum_{i=1}^N(1-u_i)\ge\prod_{i=1}^Nv_i-1+\sum_{i=1}^N(1-v_i)\qquad \mbox{(see \cite[(2.5)]{DL})}
  \end{align}
  Next we will deduce that the expression under the expectation in \eqref{drugi:skladnik} is positive.

  In order to bound this limit from above we use the inequality $u\le e^{-(1-u)}$ for $u\in\R$. Therefore, we can write
  \begin{align}
   \label{DL:nierownosc2}
   &\e{\prod_{i=1}^N\phi(e^{-x}A_i)-1+\sum_{i=1}^N(1-\phi(e^{-x}A_i))}\\\nonumber
   &\qquad\le\e{\exp\left(-\sum_{i=1}^N(1-\phi(e^{-x}A_i))\right)-1+\sum_{i=1}^N(1-\phi(e^{-x}A_i))}.
  \end{align}

  Now we will split the problem into two separate cases\\

{\bf Step 1, case i) }
  {\it  There exists constant $M$ such that for any $i$, $A_i\le M$ a.s.} Since  $\phi$ is a Laplace transform of a non-negative random variable, $(1-\phi(u))/u$ is a decreasing function whereas $(1-\phi(u))$ is increasing. Hence
  $$1-\phi(e^{-x}A_i)\le \max(A_i,1)(1-\phi(e^{-x})),$$
  and thus
  $$\sum_{i=1}^N\big[1-\phi(e^{-x}A_i)\big]
  \le(N+\sum_{i=1}^NA_i)(1-\phi(e^{-x})).$$
  Therefore, since  the function $F(u)=e^{-u}-1+u$ is increasing on $[0,\8)$, $F(u)/u$ is bounded
   and tends to 0 as $u\to0$ we can apply the Lebesgue theorem and obtain
   \begin{multline*}
    \limsup_{x\to\8}\frac{\e{\prod_{i=1}^N\phi(e^{-x}A_i)-1+\sum_{i=1}^N(1-\phi(e^{-x}A_i))}}{1-\phi(e^{-x})}\\
    \le\limsup_{x\to\8}\frac{\e{F\left(\sum_{i=1}^N(1-\phi(e^{-x}A_i))\right)}}{1-\phi(e^{-x})}
    \le    \limsup_{t\to0}\frac{\e{F\left(\left(N+\sum_{i=1}^NA_i\right)t\right)}}{t}\\
    \le    \limsup_{t\to0}\e{\frac{F\left(\left(N+\sum_{i=1}^NA_i\right)t\right)}{\left(N+\sum_{i=1}^NA_i\right)t}\cdot\left(N+\sum_{i=1}^NA_i\right)}=0,
   \end{multline*}
   because  $\e{\left(N+\sum_{i=1}^NA_i\right)}<(M+1)\e{N}<\8$.

{\bf Step 1, case ii)} {\it $A_i$'s are unbounded.}
  Take  $M>0$ big enough that will be specified later. Define a truncated random vector
  $(\wt{B}(v),\wt{A_1}(v),\dots \wt{A_N}(v))=(B(v)\wedge M,A_1(v)\wedge M,\dots,A_N(v)\wedge M)$.
  Now take
  $$m_M(t)=\e{\sum_{i=1}^N\wt{A_i}^t}$$ and observe that $m_M(t)=1$ has two different solution $\a_M<\a<\b_M$ and both of them converge to $\a$ as $M\to\8$. Indeed, it follows from the fact that $m_M(\a)<1$ and for $t\in(0,\a+\d)$ different than $\a$ the Lebesgue Theorem gives that $m_M(t)\to m(t)>1$. We will assume that $\b_M<\a+\d$.

  Define
  $$\wt{R}=\sum_{v\in\T}\wt{L}(v)\wt{B}(v),$$
  where $\wt{L}$ is defined in the same way as $L$ but in terms of $\wt{A_i}(v)$. Clearly $\wt{R}\le R$ a.s. Since $M$ can be chosen in such a way that 
  for any $j$ the measure $\pr{\log \wt{A_j}\in du, A_j>0, N\ge j }$ has the same support as $\pr{\log {A_j}\in du, A_j>0, N\ge j }$ we can apply
  Lemma \ref{thm:JO} a) for the random variable $\wt{R}$ and obtain $\P \big[{\wt{R}>t}\big]\sim C_+ t^{-\b_M}$ for some positive $C_+$. Hence, by Tauberian Theorem
  \begin{align}
  \label{asymptotyka:R:tylda}
  1-\wt{\phi}(t)\sim C_+\Gamma(1-\b_M) t^{\b_M},
  \end{align}
  where $\wt{\phi}(t)=\phi_{\wt{R}}(t)\ge\phi(t)$. Therefore we can find $C_0$ such that $1-\wt{\phi}(t)\ge C_0 t^{\b_M}$ for $0<t<1$. On the other hand, by Lemma \ref{malenie:transformaty} and Corollary \ref{cor:moments}, $1-\phi(t)\le C_1 t^{\b}$  for $\b<\a$. This implies that for some $C$ and any $t$
  $$1-\phi(t)\le C \left(1-\wt{\phi}(t)\right)^{\b/\b_M}.$$
  From \eqref{asymptotyka:R:tylda} get
  \begin{align*}
   \lim_{t\to0}\frac{1-\wt{\phi}(ts)}{1-\wt{\phi}(t)}= s^{\b_M},
  \end{align*}
   hence for any  $0<\eps<\a\d/2$ we may find $C_1$ such that
  \begin{align*}
  \frac{1-\wt{\phi}(st)}{1-\wt{\phi}(t)}\le C_1 s^{\b_M+\eps},
  \end{align*}
  for any $t>0$ and $s>1$. Therefore, we have
  \begin{align*}
   1-\phi(ts)&\le C \left(1-\wt{\phi}(ts))\right)^{\b/\b_M}\\
   &\le C\left(C_1\left(1-\wt{\phi}(t)\right)s^{\b_M+\eps}\right)^{\b/\b_M}\\
   &\le C_2\left(1-\wt{\phi}(t)\right)^{\b/\b_M}(s^{(\b_M+\eps)\frac{\b}{\b_M}}).
  \end{align*}
  Thus, since the function $F$ is increasing, $F(u)/u^{\frac{\b_M}{\b}}$ is bounded for $\b_M\le 2\b$ and $\phi(t)\le \wt \phi(t)$ we have
  {\allowdisplaybreaks\begin{align*}
   \frac{{F\left(\sum_{i=1}^N\left(1-\phi(e^{-x}A_i)\right)\right)}}{1-\phi(e^{-x})}
   &\le
   \frac{{F\left(\sum_{i=1}^N\left(1-\phi\big(e^{-x}(1\vee A_i)\big)\right)\right)}}{1-\phi(e^{-x})}\\
&\le \frac{{F\left(\left( C_2\sum_{i=1}^N(1\vee A_i^{(\b_M+\eps)\frac{\b}{\b_M}})\right)\left(1-\wt{\phi}(e^{-x})\right)^{\b/\b_M}\right)}}{1-\wt{\phi}(e^{-x})}\\
   &\!\le
  C\bigg( \sum_{i=1}^N 1\vee A_i^{(\b_M+\eps)\frac{\b}{\b_M}} \bigg)^\frac{\b_M}{\b} \\
   &\!\le
  C\bigg( N^{\frac{\b_M}{\b}} + \bigg(   \sum_{i=1}^N  A_i^{(\b_M+\eps)\frac{\b}{\b_M}} \bigg)^\frac{\b_M}{\b}\bigg).
   \end{align*}}
   Now, applying Lemma \ref{zabawa:z:momentami} with $X_i = A_i^\a$, $r = \frac{(\a+\d) \b_M}{\b(\b_M+\eps)}$ and $p=\frac{\b_M}{\b}$
(one can choose sufficiently large $M$ and $\b_M$ close to $\b$ 
such that the pair $r,p$ satisfies asumptions of Lemma \ref{zabawa:z:momentami}) we obtain
$$
\e{\frac{{F\left(\sum_{i=1}^N\left(1-\phi(e^{-x}A_i)\right)\right)}}{1-\phi(e^{-x})}
   } \le C \bigg( \E [N^{1+\d}] + \E\bigg[  \sum_{i=1}^N A_i^{\a+\d}   \bigg]
   \bigg)
$$
  Finally, by Lebesgue Theorem and since $F(u)/u\to0$ as $u\to0$ we deduce  that
  \begin{align*}
   \lim_{x\to\8}\frac{\e{F\left(\sum_{i=1}^N\left(1-\phi(e^{-x}A_i)\right)\right)}}{1-\phi(e^{-x})}=
   \e{\lim_{x\to\8}\frac{F\left(\sum_{i=1}^N\left(1-\phi(e^{-x}A_i)\right)\right)}{1-\phi(e^{-x})}}=0.
  \end{align*}
   \medskip

   {\bf Step 2.} Let us introduce a family of functions $$h_x(y)=\frac{D(x+y)}{D(x)}.$$ Dividing the equation
   $$D(x+y)=\e{D(x+y+Y)}-G(x+y)$$ by  $D(x)$, we obtain
   \begin{align}
    \label{eq:rownanie.h}
    h_x(y)=\e{h_x(y+Y)}-\frac{G(x+y)}{D(x+y)}h_x(y).
   \end{align}
   Since $D(y)e^{-\a y}=1-\phi(e^{-y})$ is decreasing and $D(y)e^{(1-\a) y}=(1-\phi(e^{-y}))/e^{-y}$ is increasing the same holds for functions $h_x$. Therefore we  conclude that $h_x(y)\le\max\{e^{\a y},e^{(\a-1)y}\}$   and that $h_x$ are equi-continuous  on bounded sets.
   By Arzel\`a-Ascoli theorem  the set $\{h_x\}$ is relatively compact in the topology of uniform convergence on compact sets.
   Take now an accumulation point $h$ as $x\to\8$ i.e. $h_{x_n}\to h$ for some sequence $x_n\to\8$. Passing to infinity with $x_n$ in \eqref{eq:rownanie.h} from Step 1 and the Lebesgue theorem we have
   \[
   h(y)=\e{h(y+Y)}.
   \]
   Since any positive $Y$-harmonic function is constant it yields that $h(y)=h(0)=1.$
   Hence, $h$ is the unique accumulation point and therefore $D(x+y)/D(x)\to1.$

\medskip

{\bf Step 3.} Finally we get rid of the additional assumption \eqref{dodatkowe.zalozenie}.
We  define $\tilde{B}=\min\{1,B\}$ and consider
$$
\wt R = \sum_{v\in\T}L(v)\tilde B(v),
$$ i.e. $\wt R$ is defined in the same way as $R$ in \eqref{rownanie:galazkowe} but with  $B(v)$ replaced by $\tilde B(v)$. Of course $\wt R \le R$, hence $\wt R$  is also finite a.s. and  solves the equation
  \begin{align}
   \tilde{X}=_d\sum_{i=1}^NA_i\tilde{X_i}+\tilde{B},
  \end{align}
  with the vector $(\tilde B ,A_1, A_2, \dots)$ satisfying hypotheses \eqref{istnieje:alpha}-\eqref{zalozenia.momenty}.

Let $\tilde{\phi}=\phi_{\tilde{R}}$ be  the Laplace transform of $\tilde R$. Notice that \eqref{dodatkowe.zalozenie} is satisfied for $R$ and $B$ replaced by $\wt R$ and $\wt B$, i.e.
  \begin{align}
  \label{dodatkowe.zalozenie2}
  \lim_{t\to 0}\frac{1-\e{e^{-t\wt B}}}{1-\wt\phi(t)}=0.
  \end{align}
Indeed,  observe first  that  $\E[\tilde R]=\8$, otherwise we would have
  $$\E[\tilde R]=\e{\sum_{i=1}^N{A_i}}\E[\tilde R]+\E[\tilde{B}]$$ which is impossible since $1<\e{\sum_{i=1}^NA_i}\le\8$. Thus in consequence  $1-\tilde{\phi}(t)/t$ tends to infinity as  $t$ goes to 0. Moreover, since $\wt B$ is bounded a.s., by the Lebesgue theorem and by Lemma \ref{malenie:transformaty} we have
  \begin{align*}
   \lim_{t\to0}\frac{1-\e{e^{-t\tilde{B}}}}{1-\tilde{\phi}(t)}=
   \lim_{t\to0}\e{\frac{1-{e^{-t\tilde{B}}}}{t}}\cdot\frac{t}{1-\tilde{\phi}(t)}=
   \E[\tilde B]\cdot\lim_{t\to0} \frac{t}{1-\tilde{\phi}(t)} =0,
  \end{align*}
  which proves  \eqref{dodatkowe.zalozenie2}.
  Therefore, we may use the results proved in the first two steps of the proof saying that
  $$
  \lim_{x\to\8} \frac{e^{\a (x+y)}(1-\wt \phi(e^{-(x+y)}))}{e^{\a x }(1-\wt \phi(e^{-x}))} = 1
  $$

Hence
$$
   \frac{1-\tilde{\phi}(t)}{t^{\alpha}}=L(1/t),
$$
  for some slowly varying function $L$.

Since $\tilde R\le R$ and $\tilde \phi \ge \phi$,  for $0<\eps<\d$ we have
  \begin{align*}
\lim_{t\to 0}   \frac{t^{\alpha+\eps}}{1-{\phi}(t)} \le
\lim_{t\to 0}\frac{t^{\alpha+\eps}}{1-\tilde{\phi}(t)} = \lim_{t\to 0}\frac{t^{\eps}}{L(1/t)} = 0.
  \end{align*}
 By Lemma \ref{malenie:transformaty}
  \begin{align*}
   \frac{1-\e{e^{-tB}}}{t^{\alpha+\eps}}\le C<\8
  \end{align*}
Finally, we get
  \begin{align*}
   \lim_{t\to0}\frac{1-\e{e^{-tB}}}{1-\phi(t)}=\lim_{t\to0}\frac{1-\e{e^{-tB}}}{t^{\a+\eps}}\cdot\frac{t^{\a+\eps}}{1-\phi(t)}=0.
  \end{align*}
 \end{proof}

The last Proposition implies immediately  the following results
 \begin{cor}
  \label{cor:twierdzenie.regularnosc}
  Under  assumptions of Theorem \ref{thm:twierdzenie.asymptotyka} we have the following: for any $s>0$
  \begin{align}
   \lim_{t\to0}\frac{1-\phi(ts)}{1-\phi(t)}=s^{\a}.
  \end{align}
 In particular, the function $L(t)=(1-\phi(1/t))t^{\a}$ is slowly varying.

 \end{cor}
\begin{corollary}
  \label{momenty:R}
  Under assumptions \eqref{istnieje:alpha}-\eqref{zalozenia.momenty} we have $\e{R^{\b}}<\8$ for $\b<\a$ and $\e{R^{\b}}=\8$ for $\b>\a$.
 \end{corollary}
 \begin{proof}
  We have
  $$\e{R^{\b}}=\b\int_0^{\8}t^{\b-1}\pr{R>t}dt=C_0+\b\int_1^{\8}t^{\b-1-\a}L(t)dt$$
  what is finite if $\b<\a$ and infinite if $\b>\a$.
 \end{proof}


\subsection{Some properties of the function $G$}
 \label{subsection3.2}

To prove our main results we will use the renewal theorem, therefore before we will proceed with the final arguments we have to prove some properties of the function $G$.


 \begin{lemma}
  \label{lemat:dRi}
  There exists $\eps>0$ such that the function $e^{\eps|x|}G(x)$ is directly Riemman integrable.
 \end{lemma}

 \begin{proof}

  For any $x$ and $\eps<\min\{\d,1-\a\}$ we have
  \begin{align*}
   e^{\pm \eps x}|G(x)|&\le e^{(\a\pm\eps) x}\e{\left|\sum_{i=1}^N (1-\phi(e^{-x} A_i))-1+\prod_{i=1}^N \phi\left(e^{-x}A_i\right)\right|}\\&\qquad+e^{(\a\pm\eps) x}\e{\left|\left(e^{-e^{-x}B}-1\right)\right|\prod_{i=1}^N \phi(e^{-x}A_i)}\\&
   =e^{(\a\pm\eps) x}\e{\sum_{i=1}^N (1-\phi(e^{-x} A_i))-1+\prod_{i=1}^N \phi\left(e^{-x}A_i\right)}\\&\qquad+e^{(\a\pm\eps) x}\e{\left(1-e^{-e^{-x}B}\right)\prod_{i=1}^N \phi(e^{-x}A_i)}\\&=f_1(x)+f_2(x),
  \end{align*}
  Notice first that since for $\g<\a+\eps$, by  Lemma \eqref{malenie:transformaty} we have
  $$ e^{\g x}\e{(1-e^{-e^{-x}B})\prod_{i=1}^N \phi(e^{-x}A_i)} \le e^{\g x}\left(1-\e{e^{-e^{-x}B}}\right)
   \le C\min\left\{e^{\g x},e^{(\g-\a-\d)x}\right\},
  $$
the function $f_2$ is directly Riemman integrable.

\medskip

  Let us now examine the  function $f_1$. First, we will prove that $f_1$ is integrable. For this purpose notice that if
   $\g>0$ we have
  \begin{align*}
   &\int_{\R}e^{\g x}\e{\sum_{i=1}^N (1-\phi(e^{-x} A_i))-1+\prod_{i=1}^N \phi(e^{-x}A_i)}dx\\
   &\qquad\qquad\le\int_{\R}e^{\g x}\e{\sum_{i=1}^N (1-\phi(e^{-x} A_i))-1+e^{-\sum_{i=1}^N (1-\phi(e^{-x}A_i))}}dx\\
   &\qquad\qquad\le\e{\int_{\R}e^{\g x}F\left(\sum_{i=1}^N (1-\phi(e^{-x} A_i))\right)dx}.
  \end{align*}
  The above expression is finite for  $\g<\a+\eps$. Indeed,
  from the monotonicity of $F$ on the positive half line and from Lemma \ref{malenie:transformaty}, for $\b<\a$  we can bound the last integral by
  \begin{align*}
   &\e{\int_{\R}e^{\g x}F\left(C\sum_{i=1}^N A_i^{\b}\e{R^{\b}}e^{-\b x} \right)dx}\\
   &\qquad\qquad=\e{\left(C\e{R^{\b}}\sum_{i=1}^N A_i^{\b}\right)^{\g/\b}}\times\int_{\R}e^{\g x}F(e^{-\b x} )dx.
  \end{align*}
   To see that the expression above is  finite we apply Lemma \ref{zabawa:z:momentami} with $r=(\a+\d)/\a$ and $X_i=A_i^{r\b}$. The second term is finite since $F(t)\le \min\{t,t^2/2\}$ and $\b<\g<2\b$. Thus $f_1$ is integrable.

   \medskip

  Next we have to check that $f_1$ satisfies \eqref{dri}.
  Inequality \eqref{DL:nierownosc1} implies that $e^{-(\a\pm\eps)x}f_1(x)$ is decreasing, hence also $e^{-x}f_1(x)$, since $\a+\eps<1$. Therefore, for any  $h>0$ and $x\in I_n(h)$ we have
    $$f_1(x)\le e^{x-nh} f_1(nh)\le e^hf_1(nh)$$
    and similarly we can estimate from below
  $$f_1(x)\ge e^{x-(n+1)h}f_1((n+1)h) \ge e^{-h}f_1((n+1)h).$$
  Since $f_1$ is integrable, the series
  $$\sum_{n\in\Z}f_1(nh)=\sum_{n\in\Z}f_1((n+1)h)\le\sum_{n\in\Z}\frac{1}{h}\int_{I_n(h)}e^hf_1(x)dx=\frac{e^h}{h}\int_{\R} f_1(x)dx<\8$$
  is finite. We can write
  \begin{align*}
   \sum_{n\in\Z}\sup_{x,y\in I_n(h)}&|f_1(x)-f_1(y)|\cdot h\le\sum_{n\in\Z}\left(e^{h}f_1(nh)-e^{-h}f_1((n+1)h)\right)\cdot h\\
   &\le\sum_{n\in\Z}\left({e^{h}-e^{-h}}\right)f_1(nh)\cdot h
   \le\sum_{n\in\Z}\left({e^{h}-e^{-h}}\right)\int_{I_n(h)}f_1(x)e^hdx\\
   &=(e^{2h}-1)\int_{\R} f_1(x)dx.
  \end{align*}
  The last expression converges to 0 as $h$ goes to 0, thus $f_1$ is directly Riemann integrable.


 \end{proof}
  \begin{corollary}
  \label{wniosek:1}
  Functions $xG(x)$ and $G(x)$ are directly Riemman integrable.
 \end{corollary}
 \begin{proof}
  The corollary follows from the fact that $(1+|x|)|G(x)|\le C\left(e^{-\eps_0x}+e^{\eps_0x}\right)|G(x)|,$ for some sufficiently large $C$.
 \end{proof}

 \begin{corollary}
  \label{wniosek:2}
  If $\int G(x)dx=0$ then the function  $\overline{G}(x)=\int^{x}_{-\8}G(y)dy$ is also dRi and satisfies $\int\overline{G}(x)dx=-\int x G(x)dx$.
 \end{corollary}
 \begin{proof}
  For $x\le0$ we have
  $$|\overline G(x)|\le \left|\int_{-\8}^xG(y)dy\right|\le \int_{-\8}^xe^{\eps_0x}e^{-\eps_0y}|G(y)|dy\le Ce^{\eps_0x}.$$
  In the same way we prove for $x>0$ that
  $$|\overline G(x)|\le \left|\int_{-\8}^xG(y)dy\right|=\left|\int_{x}^{\8}G(y)dy\right|\le Ce^{-\eps_0x}.$$
  Hence $\overline G$ is directly Riemman integrable. Moreover,
  \begin{align*}
   \int_{\R}\overline G(x)dx&=\int_{\R}\int^{x}_{-\8}G(y)dydx\\
   &=\int_{x\le0}\int^{x}_{-\8}G(y)dydx-\int_{x>0}\int_{x}^{\8}G(y)dydx\\
   &=\int_{y\le0}\int_{y}^{0}G(y)dxdy-\int_{y>0}\int_{0}^{y}G(y)dxdy=-\int_{\R}yG(y)dy.
  \end{align*}

 \end{proof}

 \subsection{Existence of the limit}
  \label{subsection3.3}

  We are able now to describe behaviour of $\phi$ near 0.
  \begin{prop}
  \label{prop:limit}
   Under hypotheses \eqref{istnieje:alpha}-\eqref{zalozenia.momenty} we have $$
   \lim_{t\to 0}\frac{1-\phi(t)}{t^\a} = C_+
   $$
  \end{prop}
  Thus, by the Tauberian theorem (Lemma \ref{thm:tauberowskie}) we deduce
  \begin{cor}
  \label{cor:limit}
   Under hypotheses \eqref{istnieje:alpha}-\eqref{zalozenia.momenty} we have $$
   \lim_{x\to \8} x^\a \P[R>x] = C_+
   $$
  \end{cor}

  \begin{proof}[Proof of Proposition \ref{prop:limit}]
   The scheme of the proof is similar to the proof of Theorem 2.18 in \cite{DL}.
   Let $\{Y_n\}$ be a sequence of independent copies of $Y$ defined in \eqref{rozklad:Y}. By $S_n$ we  denote their partial sums, i.e. $S_n = \sum_{i=1}^n Y_i$.
   We define the stopping time $ L=\inf\{n\ge0:S_n<0\}$ and the sequence of stopping times $T_k=\inf\{n>T_{k-1}:S_{n}\ge S_{T_{k-1}} \}$.

   Since $D$ is a solution of the Poisson equation the sequence of random variables
   \[
    M_n(x)=D(x+S_n)-\sum_{i=0}^{n-1}G(x+S_i)
   \]
   forms a martingale with respect to the natural filtration generated by $\{Y_n\}$. In view of the optional stopping theorem we have
   \begin{align*}
    \e{M_{n\wedge L}(x)}=\e{M_0(x)}=D(x)
   \end{align*}
and equivalently
   \begin{align*}
    \e{D(x+S_{n\wedge L})}-D(x)=\e{\sum_{i=0}^{n\wedge L-1}G(x+S_i)}.
   \end{align*}
  Next we want to pass with $n$ to infinity. Notice that  by the duality principle \cite{Feller} $$\e{\sum_{i=0}^{ L-1}|G(x+S_i)|}=\e{\sum_{i=0}^{\8}|G(x+S_{T_i})|}$$ and the last sum is finite since $G$ is dRi, hence we can pass with $n\to\8$ on the right side. In order to justify passing to limit on the left side observe that
   by Proposition \ref{prop: regularnosc D}  for any $\eps<\d$ we have $D(x)\le Ce^{\eps |x|}$. Since  $\e{e^{\eps S_{ L}}}<\8$ (see \cite{Feller} (3.6a) in Chap. XII), we can pass  to infinity.

\medskip

  Thus we obtain
   \begin{equation}
   \label{rownanie:odnowy:z:tau}
    \e{D(x+S_{ L})}-D(x)=\e{\sum_{i=0}^{ L-1}G(x+S_i)}=:R(x). 
   \end{equation}

Applying again the duality principle we  have
   \begin{align}
    R(x)=\sum_{n=0}^{\8}G(x+S_{T_n}).
   \end{align}
   Now the renewal theorem yields that
   \begin{align}
    \lim_{x\to\8}R(x)=-\frac{\int_{\R} G(x)dx}{\e{S_{T_1}}}.
   \end{align}

   Integrating \eqref{rownanie:odnowy:z:tau} we have
   \begin{align}
    \int_0^x\Big(\e{D(y+S_{ L})}-D(y)\Big)dy=\int_0^xR(y)dy
   \end{align}
   what can be rewritten as
   \begin{align}
    \label{rowanie:pomocnicze}
    D(x)\cdot\e{\int_0^{S_{ L}}\frac{D(x+y)}{D(x)}dy}-\e{\int_0^{S_{ L}}D(y)dy}=\int_0^xR(y)dy
   \end{align}
   By Proposition \ref{prop: regularnosc D}, $D(x+y)/D(x)\le C e^{\eps y}$ hence again the same argument as before tells us that we can pass to the limit under the integral sign and obtain $$\lim_{x\to\8}\e{\int_0^{S_{ L}}\frac{D(x+y)}{D(x)}dy}=\e{S_{ L}}.$$
   Dividing by $x$ in \eqref{rowanie:pomocnicze} and passing to the limit we obtain
   \begin{align*}
    \lim_{x\to\8}\frac{D(x)}{x}\e{S_{ L}}=\lim_{x\to\8}R(x)=-\frac{\int_{\R}G(x)dx}{\e{S_{T_1}}},
   \end{align*}
   and finally
   \begin{align}
    \label{granica:1}
    \lim_{x\to\8}\frac{D(x)}{x}=-\frac{\int_{\R}G(x)dx}{\e{S_{T_1}}\e{S_{ L}}}=\frac{2\int_{\R}G(x)dx}{\sigma^2},
   \end{align}
   where $\sigma^2:=\operatorname{Var} Y=2\e{-S_L}\e{S_{T_1}}$ (see the proof of T18.1 on page 196 in \cite{Spitzer}).

Notice that in view of Lemma \ref{thm:tauberowskie} the last formula implies
$$
\lim_{t\to\8} \frac{t^\a}{\log t} \P[R > t] = \frac{2\int_{\R}G(x)dx}{\sigma^2}.
$$ By Proposition \ref{prop:existence} the last constant must be 0, therefore
$$\int_{\R} G(x)dx=0.$$
Now we repeat the above procedure.  Integrating \eqref{rownanie:odnowy:z:tau} we have
   \begin{align*}
    \int_{-\8}^x\Big(\e{D(y+S_{ L})}-D(y)\Big)dy =\int_{-\8}^x\e{\sum_{i=0}^{\8}G(y+S_{T_i})}dy
   \end{align*}
   what is equivalent to
   \begin{align*}
    \e{\int_{0}^{S_{ L}}D(x+y)dy}=\e{\sum_{i=0}^{\8}\int_{-\8}^xG(y+S_{T_i})dy}
    =\e{\sum_{i=0}^{\8}\overline{G}(x+S_{T_i})}.
   \end{align*}
   Passing with $x$ to infinity in
   \begin{align*}
    D(x)\e{\int_{0}^{S_{ L}}\frac{D(x+y)}{D(x)}dy}=\e{\sum_{i=0}^{\8}\overline{G}(x+S_{T_i})}
   \end{align*}
   we obtain
   \begin{align*}
    \lim_{x\to\8}D(x)=\frac{2\int \overline{G}(x)dx}{\sigma^2}=\frac{-2\int x G(x)dx}{\sigma^2}
   \end{align*}

  \end{proof}

 \subsection{Positivity of the limiting constant}
 \label{subsection3.4}

Now we proceed with the last step of the proof and we justify that the constant $C_+$ in the statement of Corollary \ref{cor:limit} is strictly positive.
We  follow here the method developed in \cite{BDGHU}.

 The key tool to  deal with this problem is the Landau Theorem. Originally it was stated for  Dirichlet series but it can be extended for functions of the type $f(s)=\int x^s\mu(dx)$ where $\mu$ is a positive measure on $\R^+$ (compare with \cite[Theorems 5a and 5b in Chap. II]{Widder})
 \begin{theorem}[Landau]
  Let $\sigma_c$ be the abscissa of convergence for $f(z)=\int_0^{\8} x^z\mu(dx)$, i.e. the integral converges for $\Re z<\sigma_c$ and diverges for $\Re z>\sigma_c$. Then $\sigma_c$ is a singularity for $f$.
 \end{theorem}

 As a conclusion we get that if $z\mapsto\e{R^z}$ has an analytic extension on some open neighborhood of $\a$ then $\e{R^z}$ is well defined there. Our aim is to find such an extension, under the assumptions that $\int G(x)dx=0$ and $\int xG(x)dx=0$.


 \begin{theorem}
  Under assumptions \eqref{istnieje:alpha}-\eqref{zalozenia.momenty} the constant $C_+$ is strictly positive.
  \end{theorem}
\begin{proof}
 Let us define
 $$H(z)=\int_{\R} e^{xz}\e{\sum_{i=1}^N(1-\phi(e^{-x}A_i))-(1-\phi(e^{-x}))}dx.$$ By Lemma \ref{lemat:dRi}, $H$ is well defined for $0<\Re z< \a+\eps_0$. Moreover, $H$ is a holomorphic  function for $0<\Re z< \a+\eps_0$. Indeed, take any closed piecewise $C^1$ curve $\g$ in $0<\Re z<\a+\eps_0$. Then by Lemma \ref{lemat:dRi} we have
  \begin{align*}
  \int_{\g}\int_{\R} \left| e^{xz}\e{\sum_{i=1}^N(1-\phi(e^{-x}A_i))-(1-\phi(e^{-x}))}\right|dx|dz|\\
  \\\le\int_{\g}|dz|\int  (e^{-\eps_0 x}+e^{\eps_0 x})|G(x)|dx<\8.
 \end{align*}
 Hence, by the Fubini Theorem
 \begin{align*}
  \oint_{\g}H(z)dz&=\oint_{\g}\int_{\R}e^{xz}\e{\sum_{i=1}^N(1-\phi(e^{-x}A_i))-(1-\phi(e^{-x}))}dxdz\\
  &=\int_{\R}\oint_{\g}e^{xz}dz\e{\sum_{i=1}^N(1-\phi(e^{-x}A_i))-(1-\phi(e^{-x}))}dx=0.
 \end{align*}
 Therefore, by Morera Theorem $H$ is holomorphic in $0<\Re z<\a+\eps_0$ and
 \begin{align*}
  H'(\a)
 &=\lim_{h\to0}\int_{\R}\frac{e^{xh}-1}{h}e^{x\a}\e{\sum_{i=1}^N(1-\phi(e^{-x}A_i))-(1-\phi(e^{-x}))}\\
 &= \int_{\R}xe^{x\a}\e{\sum_{i=1}^N(1-\phi(e^{-x}A_i))-(1-\phi(e^{-x}))}     =\int x G(x)dx,
 \end{align*}
where the second equality follows from $({e^{xh}-1})/{h}\le Ce^{\eps |x|}$ and the Lebesgue Theorem.

%

 Now take  $z$ such that $\Re z<\a$ and observe that for any $a>0$
  \begin{align*}
  \int_{\R}e^{zx}(1-\phi&(e^{-x}a))dx=\int_{\R}e^{zx}\e{\int_0^{e^{-x}aR}e^{-u}du}dx\\
  &=\e{\int_0^{\8} e^{-u}\int_{e^{x}<aR/u}e^{zx}dxdu }=\e{\frac{a^zR^z}z\int_0^{\8} e^{-u}u^{-z}du }\\
  &=\frac{a^z\e{R^z}\Gamma(1-z)}{z}<\8.
 \end{align*}
 Therefore for $\Re z<\a$ we can express $H$ in another form:
 \begin{align*}
  H(z)&=\int_{\R}e^{zx}\e{\sum_{i=1}^N(1-\phi(e^{-x}A_i))-(1-\phi(e^{-x}))}dx\\
  &=\frac{\e{R^z}\left(\e{\sum_{i=1}^N{A_i^z}}-1\right)\Gamma(1-z)}{z}
  =\frac{\e{R^z}\left(m(z)-1\right)\Gamma(1-z)}{z}.
 \end{align*}
 From this we obtain that for $0<\Re z<\a$
  \begin{align}
  \label{rozszerzenie:holomorficzne}
  \e{R^z}=\frac{z H(z)}{(m(z)-1)\Gamma(1-z)}
 \end{align}
 However, if we assume that both $H(\a)=\int G(x)dx$ and $H'(\a)=\int x G(x)dx$ are equal to 0, then $\a$ is a root of multiplicity at least two of $H$. By our assumptions $m(\a)-1=0,$ $m'(\a)=0$. The strong convexity of $m$ yields $m''(\a)>0$. Since $\Gamma(1-\a)\neq0$
 the right side of \eqref{rozszerzenie:holomorficzne} defines a holomorphic function on $ 0<\Re z<\a+\eps_0$ that extends $\e{R^z}$. From the Landau Theorem we know that if $\e{R^z}$ has a holomorphic  extension on $\Re z<\a+\eps_0$ then it is expressed by the same formula $\e{R^z}$.  In particular, one can find positive $\d'$ such that $\e{R^{\a+\d'}}<\8$, contrary to the Corollary \ref{momenty:R}.
 \end{proof}

\bibliographystyle{plain}

\bibliography{probab}

\end{document}